\documentclass{article}
\usepackage{hyperref}    
\usepackage{graphicx}    
\usepackage{color}    
\usepackage{amsmath,amsfonts,amssymb,latexsym,amsthm,multicol}
\usepackage[english]{babel}

\usepackage{bm}
\providecommand{\B}{}
\renewcommand{\B}{\bm}
\newcommand{\D}{\partial}
\renewcommand{\le}{\leqslant}
\renewcommand{\ge}{\geqslant}

\newtheorem{theorem}{Theorem}[section]
\newtheorem{lemma}[theorem]{Lemma}
\newtheorem{assumption}{Assumption}
\theoremstyle{definition}

\title{On the stability of finite-volume schemes on~non-­uniform meshes}
\author{P.A.~Bakhvalov, M.D.~Surnachev}

\date{August 30, 2024}

\begin{document}

\numberwithin{equation}{section}
\maketitle

\begin{abstract}
In this paper, we study the L2 stability of high-order finite-volume schemes for the 1D transport equation on non-­uniform meshes. We consider the case when a small periodic perturbation is applied to a uniform mesh. For this case, we establish a sufficient stability condition. This allows to prove the $(p+1)$-th order convergence of finite-volume schemes based on $p$-th order polynomials. 
\end{abstract}


\sloppy 

\section{Introduction}

In this paper, we consider the Cauchy problem for the 1D scalar transport equation with the unit velocity and $2\pi$-periodic initial data $v_0$:
\begin{equation}
\begin{gathered}
\frac{\D v}{\D t} + \frac{\D v}{\D x} = 0, \quad 0 < t < t_{\max}, \quad x \in \mathbb{R};
\\
v(0,x) = v_0(x), \quad x \in \mathbb{R}.
\end{gathered}
\label{eq_TE}
\end{equation}
We study the $L_2$-stability of high-order linear numerical schemes for \eqref{eq_TE} on non-uniform meshes. This leaves monotone schemes (which are at most first-order) and schemes with slope limiters or other monotonization techniques out of scope of this paper.

For finite-element methods such as the standard Galerkin method and the discontinuous Galerkin method, the $L_2$-norm of the numerical solution does not increase in time. This fact is a prerequisite to any further accuracy analysis \cite{Cockburn2003,Cao2014,Cao2017}. 

In contrast, for high-order polynomial-based finite-volume (FV) methods, numerical experiments show their stability on 1D non-uniform meshes, but the theoretical justification is still pending. The situation becomes even worse on unstructured meshes, where there are only practical recommendations to enforce stability (see, for instance, \cite{Zangeneh2019}), and a bad computational mesh can cause a breakdown. 


There are two theoretical approaches to study the stability of high-order schemes on non-uniform meshes. The first one is to consider a non-uniform mesh as an image of a uniform one. Then a scheme for \eqref{eq_TE} on a non-uniform mesh becomes a scheme for the transport equation with variable velocity on a uniform mesh. This allows to apply the classical results about the stability of difference schemes for variable-coefficient equations \cite{Lax1962,Lax1966,Yanenko1969,Shintani1976}. However, there are two shortcomings on this way.


\begin{itemize}
\item These results allow the growth of the solution error as $\exp(ct)$ where $c$ does not depend on the mesh step. This growth is justified because the $L_2$-norm of a solution of the transport equation with a variable coefficient may grow in time. In contrast, in our case, the solution of \eqref{eq_TE} is $v(t,x) = v_0(x-t)$.

\item The constant $c$ in the exponent depends on the second (in \cite{Lax1962}) or higher (in \cite{Lax1966,Yanenko1969,Shintani1976}) derivatives of the mapping. Considering a checkerboard mesh (see Fig.~\ref{fig:checkerboard}) and keeping $h_{\max}/h_{\min} = const$ when refining the mesh, we get $c \gtrsim 1/h_{\max}$, which makes the stability estimate meaningless.
\end{itemize}

\begin{figure}[t]
\centering
\includegraphics[width=\linewidth]{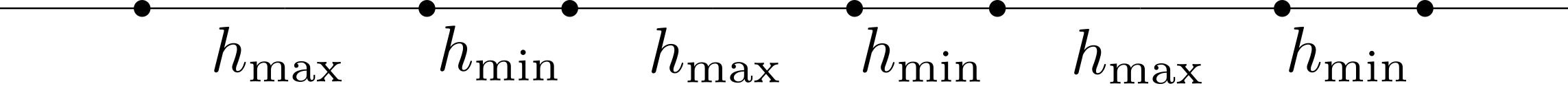}
\caption{A checkerboard mesh}\label{fig:checkerboard}
\end{figure}

The second approach, which we follow in this paper, may be used if the mesh deformation has a period of a few mesh steps (like in Fig.~\ref{fig:checkerboard}). In terms of the mapping to a uniform mesh, we have a fast-oscillating velocity coefficient. This resembles problems arising in the homogenization theory of the transport equation, see for instance \cite{E1992, Briane2020}.

If a mesh has a period, we are able to consider the period of the mesh as a ``spectral'' cell. The entire mesh becomes a uniform mesh of spectral cells. And a finite-volume scheme takes a form of a scheme with several DOFs per cell. The behavior of schemes of this class may be studied by the spectral analysis. In \cite{Zhang2005, Zhong2011, Yang2013}, this was done for the discontinuous Galerkin method (and in \cite{Zhang2005} also for the spectral volume method). In \cite{Bakhvalov2023}, an analysis was carried out for a general scheme.

In this paper, we consider a high-order finite-volume scheme for \eqref{eq_TE} on  non-uniform meshes. Our main goal is to prove a sufficient condition of its stability in $L_2$ provided that the mesh deformation is small enough. This condition is established by Theorem~\ref{th:stab}.

As a particular case, we consider the FV schemes with the polynomial reconstruction of order $p=2s$, $s \in \mathbb{N}$. They are the basis of many FV schemes for multidimensional gas dynamics \cite{Tsoutsanis2011, Antoniadis2017}. For the transport equation, we prove the error estimate showing the $(p+1)$-th order convergence. This property is sometimes referred to as supra-convergence \cite{Kreiss1986}. 



The rest of the paper is structured as follows. We state the main stability result (Theorem~\ref{th:stab}) in Section~\ref{sect:2} and prove it in Section~\ref{sect:spectral}. It is used to obtain an accuracy estimate for polynomial-based schemes in Section~\ref{sect:poly}. Finally, in Section~\ref{sect:alternating} we study the stability of three FV schemes on meshes with alternating steps and compare the results with the predictions given by Theorem~\ref{th:stab}. We also  support the accuracy estimate for the polynomial-based schemes with the numerical results.


\section{Problem formulation and the main result}
\label{sect:2}

Introduce the following notation. A {\it mesh} is a monotonically increasing sequence
$$
X = \{x_j \in \mathbb{R}, j \in \mathbb{Z}\}
$$
such that for some $N \in \mathbb{N}$ there holds \mbox{$x_{j+N} = x_j + 2\pi$} for all $j \in \mathbb{Z}$. Here $x_j$ are the {\it mesh nodes}, $N \equiv N(X)$ is the {\it number of nodes}, and $h_{av} \equiv h_{av}(X) = 2\pi/N(X)$ is the {\it average step}. For $j \in \mathbb{Z}$, denote $h_{j+1/2} = x_{j+1}-x_j$.  Put
$$
h_{\max}(X) = \max\limits_{j \in \mathbb{Z}} h_{j+1/2}, \quad h_{\min}(X) = \min\limits_{j \in \mathbb{Z}} h_{j+1/2}.
$$

A {\it mesh function} on a mesh $X$ is an $N(X)$-periodic sequence of complex numbers: $f = \{f_j \in \mathbb{C}, j \in \mathbb{Z}\}$. Equip the space of mesh functions on $X$ with the norm
\begin{equation}
\|f\|_{av} = \left(\frac{1}{N(X)} \sum\limits_{j=0}^{N(X)-1} |f_j|^2\right)^{1/2}
= \frac{1}{\sqrt{2\pi}} \left(\sum\limits_{j=0}^{N(X)-1} h_{av} |f_j|^2\right)^{1/2}.
\label{eq_norm_vper}
\end{equation}

The {\it period} of a mesh $X$ is the minimal number $m \equiv m(X)$ such that $h_{j+m+1/2} = h_{j+1/2}$ holds for each $j \in \mathbb{Z}$. Clearly, each mesh $X$  has a period $m(X) \le N(X)$. A mesh with $m = 1$ is a {\it uniform} mesh; for this mesh $h_{j+1/2} = h_{av}$ holds for each $j$.

Consider ODE systems (depending on mesh $X$) of the form
\begin{equation}
\frac{du_j(t)}{dt} + \sum\limits_{k=-S}^S a_{k}(h_{j-S+1/2}, \ldots, h_{j+S-1/2})\,u_{j+k}(t) = 0, \quad j \in \mathbb{Z}.
\label{eq1}
\end{equation}
By definition, a solution of this system is an $N(X)$-periodic sequence \mbox{$u(t) = \{u_j(t) \in \mathbb{C}, j \in \mathbb{Z}\}$} satisfying \eqref{eq1} for each $j$. 

Assume that $a_k(r_{-S},\ldots, r_{S-1})$ are real-valued, are defined in a neighborhood of  $r_{-S}=\ldots=r_S = 1$, are analytical at this point and satisfy
\begin{equation}
a_k(\alpha r_{-S},\ldots, \alpha r_{S-1}) = \frac{1}{\alpha} a_k(r_{-S},\ldots, r_{S-1})
\label{eq_invalpha}
\end{equation}
in their domain of definition.
Then \eqref{eq1} is equivalent to
\begin{equation}
\frac{du_j(t)}{dt} + \frac{1}{h_{av}}\sum\limits_{k=-S}^S a_{k}\!\left(\frac{h_{j-S+1/2}}{h_{av}}, \ldots, \frac{h_{j+S-1/2}}{h_{av}}\right)\,u_{j+k}(t) = 0.
\label{eq1a}
\end{equation}

On a uniform mesh, \eqref{eq1a} reduces to
\begin{equation}
\frac{du_j(t)}{dt} + \frac{1}{h_{av}}\sum\limits_{k=-S}^S \mathring{a}_{k} u_{j+k}(t) = 0
\label{eq1_uni}
\end{equation}
with
$$
\mathring{a}_k = a_k(1, \ldots, 1).
$$
We assume that \eqref{eq1_uni} is consistent with \eqref{eq_TE}, i. e. $\sum \mathring{a}_k = 0$, $\sum k\mathring{a}_k = 1$.

Denote
\begin{equation}
\mathring{\lambda}(\phi) = \sum\limits_{k=-S}^S \mathring{a}_{k} e^{i\phi k}.
\label{eq_lambda_phi}
\end{equation}

Let $\varkappa$ be the order of the leading dissipation term in the truncation error on the uniform meshes, i. e. for some $c \in \mathbb{C} \setminus \{0\}$  there holds
\begin{equation}
\mathrm{Re}\,\mathring{\lambda}(\phi) = c\phi^{\varkappa+1} + O(|\phi|^{\varkappa+2}) \quad \mathrm{as} \quad \phi \to 0.
\label{eq_ring_lambda}
\end{equation}
Since $\mathring{a}_k \in \mathbb{R}$, then $\varkappa$ is odd.


\begin{assumption}\label{ass:2}
The scheme is strictly dissipative, i. e. there holds
\begin{equation}
\mathrm{Re}\mathring{\lambda}(\phi) > 0, \quad \phi \in \mathbb{R}, \quad \phi/(2\pi) \not\in \mathbb{Z}.
\label{eq_strict_diss}
\end{equation}
\end{assumption}



A {\it local mapping} is a mapping $\Pi$ taking a mesh $X$ and a function $f \in C^q(\mathbb{R})$ (for some \mbox{$q \in \mathbb{N} \cup \{0\}$}) to a sequence \mbox{$\Pi_X f = \{(\Pi_X f)_j \in \mathbb{C}, j \in \mathbb{Z}\}$} of the form
$$
(\Pi_X f)_j = \langle \mu_j^{(X)}, f((\ \cdot\ + j)h_{av}) \rangle,
$$
where $\mu_j^{(X)}$ belongs to $(C^q(G))^*$ for some bounded domain $G$, is periodic w.r.t. $j$ with the period $m(X)$, invariant to the mesh scaling ($\{x_j\} \to \{x_j/k\}$, $k \in \mathbb{N}$), and satisfies $\langle \mu_j^{(X)}, 1 \rangle \ne 0$, $j \in \mathbb{Z}$. For a $2\pi$-periodic function $f$ and a local mapping $\Pi$, the sequence $\Pi_X f$ is $N(X)$-periodic, i. e. a mesh function on $X$.

An example of a local mapping is the mapping $\hat{\Pi}$ defined by
\begin{equation}
(\hat{\Pi}_X f)_j = \frac{1}{h_{j+1/2}} \int\limits_{x_{j}}^{x_{j+1}} f(x) dx, \quad j \in \mathbb{Z}.
\label{map_integral}
\end{equation}


The {\it truncation error} on $f$ in the sense of a local mapping $\Pi$ is the sequence $\epsilon_X(f, \Pi)$ with the components
$$
(\epsilon_X(f, \Pi))_j = - \left(\Pi_X \frac{df}{dx}\right)_j + \sum\limits_{k=-S}^S a_{k}\!\left(h_{j-S+1/2}, \ldots, h_{j+S-1/2}\right)\,(\Pi_X f)_{j+k}.
$$
We say that the system \eqref{eq1} is $q$-exact in the sense of $\Pi$, if for each polynomial $f$ of order $q$ there holds $\epsilon_X(f, \Pi) = 0$.

For $\mu : \mathbb{N} \to (0,\infty)$ denote by $\mathcal{F}_{\mu}$ the set of meshes
\begin{equation}
\mathcal{F}_{\mu} = \{X\ :\ h_{\max}(X)-h_{\min}(X) \le \mu(m(X))\, h_{av}(X)\}.
\label{eq_def_fmu}
\end{equation}

\begin{theorem}\label{th:stab}
Consider the scheme \eqref{eq1}, where $a_k$ are holomorphic at $(1, \ldots, 1)$ and satisfy \eqref{eq_invalpha}. Let $\varkappa$ be defined by \eqref{eq_ring_lambda}. Let the scheme
\begin{itemize}
\item satisfy Assumption~\ref{ass:2} on uniform meshes;
\item be $(\varkappa-1)$-exact in the sense of a local mapping on a family of meshes of the form \eqref{eq_def_fmu}.
\end{itemize}

Then for each $K>1$ there exists $\mu : \mathbb{N} \to (0,\infty)$ such that for each solution of \eqref{eq1} on each mesh $X \in \mathcal{F}_{\mu}$, where $\mathcal{F}_{\mu}$ is given by \eqref{eq_def_fmu}, there holds
\begin{equation}
\|u(t)\|_{av} \le K \|u(0)\|_{av}. \label{eq_9e1hio}
\end{equation}
\end{theorem}

\section{Proof of Theorem~\ref{th:stab}}
\label{sect:spectral}

Throughout this section, we consider a mesh $X$ with period \mbox{$m \equiv m(X) > 1$}.

Denote
\begin{equation}
\gamma_j = \frac{h_{j+1/2}}{h_{av}} - 1, \quad j \in \mathbb{Z},
\label{eq_def_gamma}
\end{equation}
and $\gamma = \{\gamma_j, j=0, \ldots, m-1\}$. We call $\gamma$ the {\it mesh structure}. Denote $|\gamma| = \max_j |\gamma_j|$. Obviously,
\begin{equation}
\sum\limits_{j=0}^{m-1} \gamma_j = 0.
\label{eq_sumzero}
\end{equation}

Let $\mathcal{L}(\gamma)$ be the linear operator on $\mathrm{seq}\, (\mathbb{Z},\mathbb{C})$ defined as
$$
(\mathcal{L}(\gamma) u)_j=\sum\limits_{k=-S}^S a_{k}\!\left(\gamma_{j-S}+1, \ldots, \gamma_{j+S-1}+1\right)\,u_{j+k}.
$$
Clearly, $\mathcal{L}(\gamma)$ maps an $m$-periodic sequence to an $m$-periodic one. Then the scheme \eqref{eq1a} can be written as
$$
\frac{du}{dt} + \frac{1}{h_{av}} \mathcal{L}(\gamma) u=0.
$$

Now we represent the scheme \eqref{eq1a} in the block form. Define $\mathcal{B}_m : \mathrm{seq}\,(\mathbb{Z},\mathbb{C} )\to \mathrm{seq}\,(\mathbb{Z},\mathbb{C}^m ) $ as
$$
(\mathcal{B}_m u)_\eta = (u_{\eta m}, \ldots, u_{\eta m + m-1})^T, \quad \eta \in \mathbb{Z}.
$$
On $\mathrm{seq}\,(\mathbb{Z},\mathbb{C}^m)$, define the linear operator $\mathcal{L}_m(\gamma)$ by
\begin{equation}\label{eq_m}
\mathcal{L}_m(\gamma) \mathcal{B}_m = \mathcal{B}_m \mathcal{L}(\gamma).
\end{equation}
It may be expressed as
\begin{equation}\label{eq_oper_m}
(\mathcal{L}_m(\gamma) U)_\eta =  \sum\limits_{\zeta = -\lceil S/m \rceil}^{\lceil S/m \rceil} L_\zeta(\gamma) U_{\eta+\zeta},
\end{equation}
where $L_{\zeta}(\gamma)$ is the $m \times m$ real-valued matrix with the elements
$$
(L_{\zeta}(\gamma))_{jk} = a_{\zeta m + k - j}(\gamma_{j-S}+1, \ldots, \gamma_{j+S-1}+1), \quad j,k = 0, \ldots, m-1,
$$
and we put $a_k(\ldots)=0$ for $|k| > S$.

Therefore, introducing the notation $U_{\eta} = (\mathcal{B}_m u)_\eta$, one can rewrite \eqref{eq1a} as
\begin{equation}
\frac{d U_{\eta}}{dt} + \frac{1}{h_{av}} \sum\limits_{\zeta = -\lceil S/m \rceil}^{\lceil S/m \rceil} L_{\zeta}(\gamma)\ U_{\eta + \zeta} = 0, \quad \eta \in \mathbb{Z}.
\label{eq_generalform}
\end{equation}

Systems of the form \eqref{eq_generalform} were studied in \cite{Bakhvalov2023}. Our notation is consistent with that paper as follows: $h \equiv h_{av}$, the shift between mesh blocks at $h=1$ is $T = m$, the set of degrees of freedom on a mesh block is \mbox{$M^0 = \{0, \ldots, m-1\}$}.

Let $\Pi$ be a local mapping. Denote by $v(\gamma, \phi, \Pi)$ the function taking a mesh structure $\gamma$ and $\phi \in \mathbb{C}$ to
$$
v(\gamma, \phi, \Pi) = \left\{\left[\Pi_X \exp\left(i\phi \frac{x}{h_{av}}\right)\right]_j, j = 0, \ldots, m-1\right\} \in \mathbb{C}^{m},
$$
where $X$  is the mesh with the average step $h_{av}$, structure $\gamma$, and offset $x_0 = 0$. By the definition of a local mapping, $v(\gamma, \phi, \Pi)$ does not depend on $h_{av}$.

Let $\mathcal{S}_{\phi}$ be the space of sequences \mbox{$U \in \mathrm{seq}\,(\mathbb{Z}, \mathbb{C}^m)$} such that $U_{\eta} = \exp(i\phi m\eta) U_0$, $\eta \in \mathbb{Z}$. Let $e_k$ be the vectors of the standard basis in $\mathbb{C}^m$.  Then
the vectors $e_k(\phi)$, $k=0, \ldots, m-1$, with the components
\begin{equation}
(e_k(\phi))_{\eta} = \exp(i\phi m\eta)  e_k,
\label{eq_def_ekphi}
\end{equation}
form a basis in $\mathcal{S}_{\phi}$. Denote by $L(\gamma, \phi)$ the matrix of the restriction of $\mathcal{L}_m(\gamma)$ to $\mathcal{S}_{\phi}$ in this basis:
\begin{equation}
L(\gamma, \phi) = \sum\limits_{\zeta = -\lceil S/m \rceil}^{\lceil S/m \rceil} L_{\zeta}(\gamma) \exp(i \phi m\zeta).
\label{eq_def_L_gamma_phi}
\end{equation}

Define the Fourier image of an $(N/m)$-periodic sequence $V$ with elements $V_{\eta} \in \mathbb{C}^m$  as
$$
\hat V(\phi) = \frac{m}{N}\sum_{\eta=0}^{N/m-1} \exp \left(-i m\phi \eta \right) V_\eta, \quad \phi = \frac{2\pi k}{N}, \quad k\in \mathbb{Z}.
$$
Then the inverse Fourier transform is
$$
V_\eta = \sum_{\phi} \exp\left( i m\phi \eta \right) \hat  V(\phi),
$$
where the sum is over $\phi = 0, 2\pi/N, \ldots, 2\pi/m - 1/N$.
There holds the Parseval identity
\begin{equation}
\frac{m}{N}\sum_{j=0}^{N/m-1} \|V_j\|^2 = \sum_{\phi} \|\hat{V}(\phi)\|^2.
\label{eq_Parseval}
\end{equation}
Here and below, by  $\|\ \cdot\ \|$ we denote the Euclidean norm on $\mathbb{C}^m$.

In particular, we define
$$
\hat U(t,\phi) = \frac{m}{N}\sum_{\eta=0}^{N/m-1} \exp \left(-i m\phi \eta \right) U_\eta(t), \quad \phi = \frac{2\pi k}{N}, \quad k\in \mathbb{Z}.
$$

In the Fourier domain, the system \eqref{eq_generalform} takes the form
\begin{equation}\label{Fourier_Im_Eq}
\frac{d}{dt} \hat U(t,\phi) + \frac{1}{h_{av}} L(\gamma,\phi) \hat U(t,\phi)=0, \quad \phi = \frac{2\pi k}{N}, \quad k\in \mathbb{Z}.
\end{equation}
The vector function $\hat U(t,\phi)$ and the matrix function $L(\gamma,\phi)$ have the period $2\pi/m$ with respect to $\phi$, so in \eqref{Fourier_Im_Eq} we have $N/m$ independent equations.

The functions $v(\gamma, \phi, \Pi)$ and $L(\gamma, \phi)$ are defined for $\gamma$ satisfying $\sum \gamma_j = 0$ with sufficiently small $\vert \gamma \vert$. Extend them to a neighborhood of zero by 
$$
v(\gamma, \phi, \Pi) = v\left(\gamma - \frac{\B{1}}{m}\sum_j \gamma_j, \phi, \Pi\right),
\quad
L(\gamma, \phi) = L\left(\gamma- \frac{\B{1}}{m}\sum_j \gamma_j, \phi\right),
$$
where $\B{1} = (1, \ldots, 1)^T$.
Introduce
$$
A(\gamma, \phi) = i\phi I - L(\gamma, \phi),
$$
$$
\hat{\epsilon}(\gamma, \phi, \Pi) = A(\gamma, \phi) v(\gamma, \phi, \Pi),
$$

\begin{lemma}\label{th:lemma3}
Let $\Pi$ be a local mapping, $X$ be a mesh, $h_{av}$ and $\gamma$ be its average step and structure, $\alpha \in \mathbb{Z}$, $t \ge 0$. Then there holds
\begin{equation}
\|\epsilon_X(\exp(i \alpha x), \Pi)\|_{av} = \frac{1}{\sqrt{m}\, h_{av}} \|\hat{\epsilon}(\gamma, \alpha h_{av}, \Pi)\|.
\label{eq_fourier_eps}
\end{equation}
\end{lemma}
\begin{proof}
Denote $\phi  = \alpha h_{av}$. Since
$$
\Pi_X e^{i\alpha x} \in S_{\phi}, \quad (\mathcal{B}_m \Pi_X e^{i\alpha x})_0 = v(\gamma,\alpha h_{av},\Pi),
$$
we have $\epsilon_X (\exp (i\alpha x),\Pi) \in S_{\phi}$ and
$$
(\mathcal{B}_m \epsilon_X (e^{i\alpha x},\Pi))_0= \frac{1}{h_{av}}A(\gamma,\alpha h_{av})v(\gamma,\alpha h_{av},\Pi) = \frac{1}{h_{av}}  \hat{\epsilon}(\gamma, \alpha h_{av}, \Pi).
$$
It remains to use \eqref{eq_Parseval}.
\end{proof}

\begin{lemma}\label{th:L:ev}
Let $S_m(\phi)$ be the matrix with the elements
\begin{equation}
(S_m(\phi))_{jk} = m^{-1/2} \exp(2\pi i j k/m + i\phi j), \quad j, k = 0, \ldots, m-1.
\label{eq_def_Smphi}
\end{equation}
Then $S_m(\phi)$ is unitary and
\begin{equation}
L(0,\phi) = S_m(\phi)\ \mathrm{diag}\{\mathring{\lambda}(\phi + 2\pi l/m),\ l=0,\ldots,m-1\}\ S_m^{-1}(\phi).\label{eq_Lev_toprove}
\end{equation}
\end{lemma}
\begin{proof}
The proof that $S_m(\phi)$ is unitary is straightforward.

Denote by $Y(\phi)$ the right-hand side of \eqref{eq_Lev_toprove}.
Using the expression \eqref{eq_lambda_phi} for $\mathring{\lambda}(\phi)$ we get
\begin{equation*}
\begin{gathered}
Y_{jk}(\phi) =
\sum\limits_{p=-S}^S \mathring{a}_p \exp(i\phi(j+p-k)) \left[\frac{1}{m} \sum\limits_{l=0}^{m-1} \exp\left(2\pi i l\frac{j+p-k}{m}\right)\right] =
\\
= \sum\limits_{\zeta \in \mathbb{Z}} \mathring{a}_{k-j+\zeta m} \exp(i\phi m \zeta).
\end{gathered}
\end{equation*}
Now it is easy to see that $Y_{jk}(\phi)$ coincides with $(L(0,\phi))_{jk}$ defined by \eqref{eq_def_L_gamma_phi}.

There is another way to see \eqref{eq_Lev_toprove} without explicit calculations. Let $\tilde S_\phi\subset \mathrm{seq}(\mathbb{Z},\mathbb{C})$ be the space of sequences such that $u_{k+m\eta} = e^{im\eta \phi} u_k$ and its basis be formed by the sequences $E_k(\phi)$, $k=0,\ldots,m-1$, such that $(E_k(\phi))_j=\delta_{kj}$, $j=0,\ldots,m-1$. The matrix $L(0,\phi)$ is the matrix of the restriction of the operator $\mathcal{L}(0)$ to the space $\tilde S_\phi$ in the basis $E_k(\phi)$. The sequences $m^{-1/2} w(\phi + 2\pi k/m)$, $k=0,\ldots,m-1$ belong to $\tilde S_\phi$ and are eigenvectors of $\mathcal{L}(0)$ with eigenvalues 
 $\mathring{\lambda}(\phi + 2\pi k/m)$. In the basis $E_k(\phi)$ these vectors have coordinates which are written in the columns of the matrix $S_m(\phi)$.
\end{proof}

Denote by $\mathbb{C}^{n \times n}$ the space of complex-valued matices of size $n \times n$, and by $\mathcal{A}(\mathbb{C}^d, \mathbb{C}^{n \times n})$ the space of functions from $\mathbb{C}^d$ to $\mathbb{C}^{n \times n}$ holomorphic at zero.
By $\sigma(A)$ we denote the spectrum of $A$.

We need the following simple result of the theory of perturbations \cite{Kato1966, Baumgartel1985}.

\begin{lemma}
Let $A \in \mathcal{A}(\mathbb{C}^d, \mathbb{C}^{n \times n})$. Let $\sigma(A(0)) = G \cup H$, \mbox{$G \cap H = \emptyset$}.
Then in a neighborhood of $\B{x}=0$ there holds
\mbox{$A(\B{x}) = S(\B{x})M(\B{x})S^{-1}(\B{x})$}, where
$S, M, S^{-1} \in \mathcal{A}(\mathbb{C}^d, \mathbb{C}^{n \times n})$,
\begin{equation}
M(\B{x}) = \left(\begin{array}{cc} M^{(G)}(\B{x}) & 0 \\ 0 & M^{(H)}(\B{x}) \end{array}\right),
\label{eq_rep}
\end{equation}
$\sigma(M^{(G)}(0)) = G$, $\sigma(M^{(H)}(0)) = H$, and the size $n_G$ of the matrix $M^{(G)}(\B{x})$ equals to the sum of orders of eigenvalues in $G$.
\label{th:mainmatrix}
\end{lemma}
The proof of Lemma~\ref{th:mainmatrix} in exactly this form is given in \cite{Bakhvalov2019beng}.

The function $L(\gamma, \phi)$ is a holomorphic function of $m+1$ variables ($\gamma_0, \ldots, \gamma_{m-1}$, $\phi$) at each point $(0, \phi)$, $\phi \in \mathbb{R}$. In particluar, it is holomorphic at $(0,0)$. Under Assumption~\ref{ass:2}, by Lemma~\ref{th:L:ev} the matrix $L(0,0)$ has zero eigenvalue of order 1. By Lemma~\ref{th:mainmatrix}, in a neighborhood of zero there holds
\begin{equation}
L(\gamma, \phi) = \tilde{S}(\gamma, \phi) \left(\begin{array}{cc}  \lambda_0(\gamma, \phi) & 0 \\ 0 & \tilde{M}^{(*)}(\gamma, \phi)  \end{array}\right) \tilde{S}^{-1}(\gamma, \phi),
\label{eq_matrix_decomp0}
\end{equation}
where the matrices $\tilde{S}$, $\tilde{M}^{(*)}$, $\tilde{S}^{-1}$ and the scalar function $\lambda_0$ are holomorphic at zero, and $\tilde{M}^{(*)}(0,0)$ is nondegenerate. Obviously, in a (generally smaller) neighborhood of $\phi=0$ there holds $\lambda_0(0,\phi) = \mathring{\lambda}(\phi)$. Let $S_m(\phi)$ be the unitary matrix defined by \eqref{eq_def_Smphi}. Both matrices $\tilde{S}(0,\phi)$ and $S_m(\phi)$ transform $L(0,\phi)$ to the block-diagonal form, so there holds
$$
\tilde{S}(0,\phi) = S_m(\phi) \left(\begin{array}{cc}  a(\phi) & 0 \\ 0 & T(\phi)  \end{array}\right)
$$
with some $a(\phi) \in \mathbb{C} \setminus \{0\}$ and a nondegenerate matrix $T(\phi)$.
Introducing $S(\gamma, \phi) = \tilde{S}(\gamma, \phi)  \tilde{S}^{-1}(0,\phi) S_m(\phi)$ we come to the decomposition
\begin{equation}
L(\gamma, \phi) = S(\gamma, \phi) \left(\begin{array}{cc}  \lambda_0(\gamma, \phi) & 0 \\ 0 & M^{(*)}(\gamma, \phi)  \end{array}\right) S^{-1}(\gamma, \phi).
\label{eq_matrix_decomp}
\end{equation}
We have $S(0,\phi) = S_m(\phi)$ and hence for each $\tilde{K}>1$ there exists a neighborhood of zero such that
\begin{equation}
\|S(\gamma, \phi)\| \le \tilde{K}, \quad \|S^{-1}(\gamma, \phi)\| \le \tilde{K}.
\label{eq_matrix_decomp_s}
\end{equation}
Besides, $M^{(*)}(0, \phi)$ is a diagonal matrix.

\begin{lemma}\label{th:expsum}
Let $B$ and $D$ be square matrices. Let $D$ be a normal matrix, and its eigenvalues $d_j$ satisfy $\mathrm{Re}\,d_j \le \mu$. Then
$$
\|\exp(B+D)\| \le \exp(\mu+\|B\|).
$$
\end{lemma}
\begin{proof}
By the Lie formula,
$$
\exp(B+D) = \lim\limits_{n \rightarrow \infty} \left(\exp(B/n) \exp(D/n)\right)^n.
$$
Hence
$$
\|\exp(B+D)\|  \le
\lim\limits_{n \rightarrow \infty} \left\| \exp(B/n)\right\|^n\
\lim\limits_{n \rightarrow \infty} \left\| \exp(D/n)\right\|^n.
$$
Since $\|\exp(B/n)\| \le \exp(\|B\|/n)$, the first multiplier in the right-hand side does not exceed $\exp(\|B\|)$. Since the matrix $D/n$ is normal it can be transformed to a diagonal matrix by a unitary transform. Thus, $\|\exp(D/n)\|$ equals the exponent of the maximal real part of the eigenvalues of $D/n$. This yields the statement of the lemma.
\end{proof}


By Assumption~\ref{ass:2}, the function $\mathrm{Re}\mathring{\lambda}(\phi)$ is a bijection
\begin{equation}
\mathrm{Re}\mathring{\lambda}(\phi)\ :\ [0, \phi_{\max}] \rightarrow [0, x_{\max}]
\label{eq_def_phimax}
\end{equation}
for some $\phi_{\max}>0$ and $x_{\max}>0$ such that
\begin{equation}
\mathrm{Re}\mathring{\lambda}(\phi) \ge x_{\max}, \quad \phi_{\max} \le \phi \le 2\pi - \phi_{\max}.
\label{eq_def_phimax_1}
\end{equation}

\begin{lemma}
\label{th:aux:9}
Let $\mathrm{Re}\,\lambda_0(\gamma, \phi) \ge 0$ hold in a neighborhood of $(0,0)$. Then for each $K > 1$ there exists $\gamma_{\max}>0$ such that for each $\phi \in \mathbb{R}$, \mbox{$|\gamma| \le \gamma_{\max}$}, and $\nu \ge 0$ there holds
\begin{equation}
\|\exp(-\nu L(\gamma, \phi))\| \le K. \label{aux:th9:1}
\end{equation}
\end{lemma}

If $\lambda_0(0,\phi) \equiv \mathring{\lambda}(\phi)$ is a contour without  self-intersections and $\inf \vert \mathring{\lambda}'(\phi)\vert \ne 0$, this lemma is obvious. We give a proof for a general case.

\begin{proof}
Let $\phi_{\max}$ and $x_{\max}$ be defined by \eqref{eq_def_phimax}--\eqref{eq_def_phimax_1}. The decomposition \eqref{eq_matrix_decomp}, inequalities \eqref{eq_matrix_decomp_s} with $\tilde{K} = K^{1/2}$, and the inequality $\mathrm{Re}\,\lambda_0(\gamma, \phi) \ge 0$ hold in a parallelepiped of the form
$$
\left\{(\gamma, \phi)\ :\ \gamma \in [-\tilde{\gamma}, \tilde{\gamma}]^m, \quad \phi \in [-\tilde{\phi}, \tilde{\phi}]\right\}
$$
with some $\tilde{\gamma}, \tilde{\phi} > 0$.
Without loss, $\tilde{\phi} \le \phi_{\max}$. Put \mbox{$\tilde{x} = \mathrm{Re}\mathring{\lambda}(\tilde{\phi})$}.

First consider the case $|\phi| \le \tilde{\phi}$. From
\eqref{eq_matrix_decomp} we have
$$
\|\exp(-\nu L(\gamma, \phi))\| \le \tilde{K}^2 \max\{|\exp(-\nu \lambda_0(\gamma, \phi))|, \|\exp(-\nu M^*(\gamma, \phi))\|\}.
$$
The first argument of the maximum is not greater than one by assumption. To estimate the second argument, use Lemma~\ref{th:expsum} with $D = -\nu M^{(*)}(0, \phi)$ and \mbox{$B = -\nu(M^{(*)}(\gamma, \phi)-M^{(*)}(0, \phi))$}. All eigenvalues of $M^{(*)}(0, \phi)$ have the real part greater than or equal to $x_{\max} > 0$. Then there exists $\hat{\gamma}_{\max}>0$ such that for $|\gamma| \le \hat{\gamma}_{\max}$ there holds $\|B\| \le \nu x_{\max}$, so the second argument of the maximum is not greater than unit either. Thus we get \eqref{aux:th9:1}.

Now consider the case $|\phi| > \tilde{\phi}$, $|\phi| \le \pi/m$. The real part of an eigenvalue of $L(0,\phi)$ is greater than or equal to $\tilde{x}$. Denote \mbox{$D = -\nu L(0,\phi)$} ($D$ is a normal matrix) and $B = -\nu(L(\gamma, \phi) - L(0,\phi))$. Then there exists $\hat{\gamma}_{\max}>0$ such that for $|\gamma| \le \hat{\gamma}_{\max}$ there holds $\|B\| \le \nu \tilde{x}$.
By Lemma~\ref{th:expsum} we get \eqref{aux:th9:1}.

We proved \eqref{aux:th9:1} for $|\phi| \le \pi/m$. Since $L(\gamma, \phi)$ is $2\pi/m$-periodic with respect to $\phi$, then inequality \eqref{aux:th9:1} holds for each $\phi \in \mathbb{R}$.
\end{proof}


In Lemma~\ref{th:aux:9} we assume that $\mathrm{Re}\,\lambda_0(\gamma, \phi) \ge 0$ in a neighborhood of $(0,0)$. Now we investigate when this assumption holds.

Let $\varkappa$ be defined by \eqref{eq_ring_lambda}, and denote by $q+1$ the minimal power of $\phi$ in the Taylor expansion of $\lambda_0(\gamma, \phi)-i\phi$ at $(0,0)$. Recall that $\lambda_0(0,\phi) = \mathring{\lambda}(\phi)$, and $\mathring{\lambda}(\phi)$ is given by \eqref{eq_lambda_phi}.

\begin{lemma}
Let $q \ge \varkappa-1$. Then in a neighborhood of $(0,0)$ there holds $\mathrm{Re}\,\lambda_0(\gamma, \phi) \ge 0$.
\label{th:aux:07}
\end{lemma}
\begin{proof}


By assumption,
$$
\lambda_0(\gamma,\phi) = i\phi + c\phi^{\varkappa+1} + (i\phi)^{\varkappa} (c_1(\gamma) + c_2(\gamma) \phi) + O(|\phi|^{\varkappa+2}),
$$
where $c>0$ is the same as in \eqref{eq_ring_lambda}, $c_1 = O(|\gamma|)$, and $c_2 = O(|\gamma|)$. We need to prove that  $c_1(\gamma) \in \mathbb{R}$ in a neighborhood of $\gamma=0$, then the statement of the lemma will be obvious.

The matrix $L(\gamma, i\phi)$ is real-valued by construction. For a sufficiently small $|\gamma|$ and $|\phi|$, by \eqref{eq_matrix_decomp}, the only eigenvalue of $L(\gamma, i\phi)$ in a sufficiently small ball $B_{\epsilon}(0)$ is $\lambda_0(\gamma, i\phi)$. If \mbox{$\lambda_0(\gamma, i\phi) \not\in \mathbb{R}$}, then $L(\gamma, i\phi)$ would have another eigenvalue, $\bar{\lambda}_0(\gamma, i\phi)$, also belonging to this ball. Thus $\lambda_0(\gamma, i\phi) \in \mathbb{R}$. And hence, all derivatives of $\lambda_0(\gamma, i\phi)$ with respect to $\phi$ are also real-valued.
\end{proof}

It remains to specify, how to detect $q$ without the explicit use of $L(\gamma, \phi)$.

\begin{lemma}
Let $\Pi$ be a local mapping. Let the truncation error in the sense of $\Pi$ satisfy
\begin{equation}
\|\epsilon_X(f, \Pi)\|_{av} \le C(\gamma)\ (h_{av})^q |f|_{q+1},
\label{eq_aux111_000}
\end{equation}
for some $q$ and $C(\gamma)$, where $h_{av}$ and $\gamma$ are the average step and the structure of $X$. Then in a neighborhood of $(0,0)$ there holds
\begin{equation}
|\lambda_0(\gamma,\phi) - i\phi| \le 2 C(\gamma)\  |\phi|^{q+1},
\label{eq_aux111_a}
\end{equation}
with the same $C(\gamma)$.
\label{th:aux:111}
\end{lemma}
\begin{proof}
By \eqref{eq_fourier_eps}, in a neighborhood of $\phi=0$ there holds
$$
\|\hat{\epsilon}(\gamma, \phi, \Pi)\|
\le \sqrt{m} C(\gamma)\ |\phi|^{q+1}.
$$
Using the decomposition \eqref{eq_matrix_decomp}, write
\begin{equation*}
\begin{gathered}
S^{-1}(\gamma, \phi) \hat{\epsilon}(\gamma, \phi, \Pi) = \\
=\left(\begin{array}{cc}  i\phi - \lambda_0(\gamma, \phi)  & 0
\\
0 & i\phi I - M^{(*)}(\gamma, \phi)  \end{array}\right) S^{-1}(\gamma, \phi) v(\gamma, \phi, \Pi).
\end{gathered}
\end{equation*}

Since $v(\gamma, \phi, \Pi)$ tends to $(1,\ldots, 1)^T$ as $\phi \rightarrow 0$, and \mbox{$S^{-1}(\gamma, \phi) \rightarrow S_m(0)$} as $(\gamma, \phi) \rightarrow (0,0)$, then the first component of $S^{-1}(\gamma, \phi) v(\gamma, \phi, \Pi)$ tends to $\sqrt{m}$. On the other hand, the norm of the left-hand side in a neighborhood of zero does not exceed $\sqrt{2m}C(\gamma) |\phi|^{q+1}$. From here we have \eqref{eq_aux111_a}.
\end{proof}



Now we are ready to prove the general stability result. 
\begin{proof}[Proof of Theorem~\ref{th:stab}]
If $m=1$, then the mesh is uniform. For any $K$, put $\mu(1) = 1$. By Assumption~\ref{ass:2} we have $\|u(t)\|_{av} \le \|u(0)\|_{av}$.

Let $m \in \mathbb{N} \setminus \{1\}$. By Lemma~\ref{th:aux:111} there holds \eqref{eq_aux111_a}. The assumptions of Lemma~\ref{th:aux:07} are satisfied, so there holds $\mathrm{Re}\,\lambda_0(\gamma, \phi) \ge 0$ in a neighborhood of $(0,0)$. For each $K>1$, put $\mu(m) = \gamma_{\max}$ where $\gamma_{\max}$ is given by Lemma~\ref{th:aux:9}. Thus, for each $X \in \mathcal{F}_{\mu}$ there holds \eqref{aux:th9:1}. 

It remains to see that for each solution $u(t)$ of \eqref{eq1} the inequality \eqref{eq_9e1hio} follows from \eqref{aux:th9:1} and the Parseval identity \eqref{eq_Parseval}.
\end{proof}

In the notation of \cite{Bakhvalov2023} under the assumptions of Theorem~\ref{th:stab} and a fixed mesh structure, the scheme \eqref{eq1} is ``simple'' and possesses the long-time simulation order $q$.

\section{The schemes with the polynomial reconstruction}
\label{sect:poly}

In this section we prove an error estimate for the FV schemes with a polynomial reconstruction for  \eqref{eq_TE}.

Let $p = 2s$, $s \in \mathbb{N} \cup \{0\}$, be the order of the polynomial in use for the variable reconstruction. Let $X = \{x_j\}$ be a mesh. By $\hat{\Pi}$ we denote the integral mapping \eqref{map_integral}. The finite-volume scheme with the $p$-order polynomial reconstruction has the form
\begin{equation}
\frac{du_j}{dt} + \frac{p_{j}(x_{j+1}) - p_{j-1}(x_{j})}{h_{j+1/2}} = 0,
\label{eqFV1}
\end{equation}
\begin{equation}
u_j(0) = \frac{1}{h_{j+1/2}} \int\limits_{x_{j}}^{x_{j+1}} v_0(x) dx,
\label{eqFV2}
\end{equation}
where $p_j(x)$ is the $p$-th order polynomial such that
\begin{equation}
\frac{1}{h_{j+k+1/2}} \int\limits_{x_{j+k}}^{x_{j+k+1}} p_j(x) dx = u_{j+k}
\label{eqFV3}
\end{equation}
holds for each $k = -p/2, \ldots, p/2$. This system has $p+1$ equations and $p+1$ unknowns; its consistency is well-known (see \cite{Shu1997}).

On the uniform mesh $X = \{jh, j \in \mathbb{Z}\}$, the scheme \eqref{eqFV1}, \eqref{eqFV2}, \eqref{eqFV3} reduces to the finite-difference scheme of order $p+1$ of the form
\begin{equation}
\frac{du_j(t)}{dt} + \frac{1}{h}\sum\limits_{k=-s-1}^s \mathring{a}_{k} u_{j+k}(t) = 0
\label{eq1_uni_a}
\end{equation}
with the initial data $u(0) = \hat{\Pi}_X v_0$.
Behavior of this scheme is well known. In particular, the truncation error of this scheme has the representation (see Theorem~2 in \cite{Iserles1982})
\begin{equation}
(\epsilon_X(f, \hat{\Pi}))_j = c_s (-1)^{s+1} h^{2s+1} \frac{d^{2s+2} f}{dx^{2s+2}}((j+\theta) h),\quad -s-1 \le \theta \le s,
\label{eq_aux_2}
\end{equation}
with
\begin{equation}
c_s = \frac{s! (s+1)!}{(2s+2)!},
\label{eq_def_cs}
\end{equation}
and the function $\mathring{\lambda}(\phi)$ defined by \eqref{eq_lambda_phi} satisfies
\begin{equation}
\mathrm{Re}\mathring{\lambda}(\phi) = (2^s c_{s})  \sin^{2s+2}(\phi/2).
\label{eq_re_lambda}
\end{equation}
From here, $\mathrm{Re}\mathring{\lambda}(\phi) > 0$ for $\phi/(2\pi) \not\in \mathbb{Z}$. Thus Assumption~\ref{ass:2} holds.


Let $X$ be a mesh. For a function $f \in C(\mathbb{R})$, let $p_j[f]$ be the polynomial defined by \eqref{eqFV3} with $u_{j+k} = (\hat{\Pi}_X f)_{j+k}$.

\begin{lemma}
Let $g(x) = f(x+mh_{av})$. Then $p_{j+m}[f](x+mh_{av}) = p_j[g](x)$.
\end{lemma}
\begin{proof}
The system \eqref{eqFV3} to define $p_{j+m}[f]$ has the form
\begin{equation*}
\int\limits_{x_{j+m+k}}^{x_{j+m+k+1}} p_{j+m}(x) dx = \int\limits_{x_{j+m+k-1}}^{x_{j+m+k+1}} f(x) dx, \quad k = -p/2, \ldots, p/2.
\end{equation*}
Since $x_{j+m+k} = x_{j+k} + mh_{av}$ and $x_{j+m+k+1} = x_{j+k+1} + mh_{av}$, it reduces to
$$
\int\limits_{x_{j+k}}^{x_{j+k+1}} p_{j+m}(x+mh_{av}) dx = \int\limits_{x_{j+k}}^{x_{j+k+1}} f(x+mh_{av}) dx, \quad k = -p/2, \ldots, p/2,
$$
which coincides with the system for $p_j[g]$.

\end{proof}

\begin{lemma}
\label{th:eps_zeroav}
There holds
\begin{equation}
\sum\limits_{j=0}^{m-1} h_{j+1/2} (\epsilon_X(x^{p+1},\hat{\Pi}))_j = 0.
\label{eq_eps_zeroav}
\end{equation}
\end{lemma}
\begin{proof}
By construction,
$$
(\epsilon_X(x^{p+1},\hat{\Pi}))_j = \frac{1}{h_{j+1/2}}(\Delta_{j+1}[x^{p+1}] - \Delta_{j}[x^{p+1}])
$$
with $\Delta_{j}[f] = p_{j-1}[f](x_{j}) - f(x_{j})$. 
By construction, $\Delta_{j}[f]$ has the following properties:
\begin{enumerate} 
\item $\Delta_{j}[f]$ is linear with respect to $f$;
\item $\Delta_{j}[f] = 0$ if the function $f$ is a polynomial of order not greater than $p$;
\item if $g(x) = f(x+mh_{av})$, then $\Delta_{j+m}[f] = \Delta_{j}[g]$.
\end{enumerate}
Using these properties we see that 
\begin{equation*}
\begin{gathered}
\Delta_{m}[x^{p+1}] =\Delta_{0}[(x+mh_{av})^{p+1}] = 
\\
\Delta_{0}[x^{p+1}] + \Delta_{0}[(x+mh_{av})^{p+1}-x^{p+1}]=\Delta_{0}[x^{p+1}].
\end{gathered}
\end{equation*}
From here, \eqref{eq_eps_zeroav} is obvious.
\end{proof}

\begin{lemma}
\label{th:aux:10}
There exists $\mu\ :\ \mathbb{N} \to (0,\infty)$ and a local mapping $\tilde{\Pi}$ of the form
\begin{equation}
(\tilde{\Pi}_X f)_j = \frac{1}{h_{j+1/2}} \int\limits_{x_{j}}^{x_{j+1}} f(x) dx + \mathfrak{C}_j^{(X)} (h_{av}(X))^{p+1} \frac{d^{p+1} f}{dx^{p+1}}(x_j),
\label{eq_tildepih}
\end{equation}
where $\mathfrak{C}_j^{(X)}$ is an $m(X)$-periodic real-valued sequence such that the system \eqref{eqFV1} is $(p+1)$-exact on each mesh $X$ from a family $\mathcal{F}_{\mu}$ given by \eqref{eq_def_fmu}. Besides, \mbox{$|\mathfrak{C}_j^{(X)}| \le c_1(h_{\max}-h_{\min})/h_{av}$} and $c_1$ depends only on $p$ and $m(X)$.
\end{lemma}
\begin{proof}
Since the system \eqref{eqFV1} is $p$-exact in the sense of $\hat{\Pi}$ by construction, it is $p$-exact in the sense of each $\tilde{\Pi}$ of the form \eqref{eq_tildepih}. By linearity, to prove $(p+1)$-exactness, it is enough to consider only $f(x) = x^{p+1}/(p+1)!$.

There holds
$(\tilde{\Pi}_X f')_j = (\hat{\Pi}_X f')_j$ and $(\tilde{\Pi}_X f)_j = (\hat{\Pi}_X f)_j + h_{av}^{p+1} \mathfrak{C}_j^{(X)}$. Then
\begin{gather*}
(\epsilon_X(f, \tilde{\Pi}))_j = -
(\hat{\Pi}_X f')_j 
+
\\
+
\frac{1}{h_{av}} \sum\limits_{k=-s-1}^s a_{k}\!\left(\frac{h_{j-S+1/2}}{h_{av}}, \ldots, \frac{h_{j+S-1/2}}{h_{av}}\right)\left[(\hat{\Pi}_X f)_{j+k} + h_{av}^{p+1} \mathfrak{C}_{j+k}^{(X)}\right]
=
\\
=
(\epsilon_X(f, \hat{\Pi}))_j +
\sum\limits_{k=-s-1}^s a_{k}\!\left(\frac{h_{j-S+1/2}}{h_{av}}, \ldots, \frac{h_{j+S-1/2}}{h_{av}}\right) h_{av}^{p} \mathfrak{C}_{j+k}^{(X)}.
\end{gather*}
Equating the right-hand side to zero, we get the system
$$
\mathcal{L}(\gamma) \mathfrak{C}^{(X)} = -h_{av}^{-p}\epsilon_X(f,\hat{\Pi})
$$
with respect to $\mathfrak{C}^{(X)} = \{\mathfrak{C}_j^{(X)}, j \in \mathbb{Z}\}$.
For $m(X)$-periodic $\mathfrak{C}^{(X)}$, this system reduces to
\begin{equation}
L(\gamma, 0) \left(\begin{array}{c} \mathfrak{C}_0^{(X)} \\ \vdots \\ \mathfrak{C}_{m-1}^{(X)} \end{array}\right) = -h_{av}^{-p} \left(\begin{array}{c} (\epsilon_X(f,\hat{\Pi}))_0 \\ \vdots \\ (\epsilon_X(f,\hat{\Pi}))_{m-1} \end{array}\right).
\label{sys_for_cx}
\end{equation}

By construction, the scheme \eqref{eqFV1} is conservative, i. e. for each $N$-periodic sequence  $u$ there holds
\begin{equation}\label{cons0}
\sum_{j=0}^{N-1}h_{j+1/2} (\mathcal{L}(\gamma) u)_j =0.
\end{equation}
For $m(X)$-periodic sequences the property \eqref{cons0} takes the form
\begin{equation*}
\sum_{j=0}^{m-1}h_{j+1/2} (\mathcal{L}(\gamma) u)_j =0.
\end{equation*}
From here, recalling \eqref{eq_m}, \eqref{eq_oper_m}, \eqref{eq_def_L_gamma_phi}, we get
$$
\sum_{j=0}^{m-1}h_{j+1/2} (L(\gamma,0) U)_j =0 \quad \text{for all}\quad U \in \mathbb{C}^m,
$$
which means that $(h_{1/2}, \ldots, h_{m-1/2})$ is the left eigenvector of $L(\gamma, 0)$ corresponding to zero eigenvalue. The spectrum of $L(0,0)$ is given by Lemma~\ref{th:L:ev}. By \eqref{eq_re_lambda}, zero is a simple eigenvalue of $L(0,0)$. By continuity, it remains simple for sufficiently small $|\gamma|$. Hence, the consistency of \eqref{sys_for_cx} is equivalent to \eqref{eq_eps_zeroav} and holds by Lemma~\ref{th:eps_zeroav}.

Since $L(\gamma,0)$ is continuous in $\gamma$, and $L(0,0)$ is a circulant matrix, then for a sufficiently small $|\gamma|$ the system \eqref{sys_for_cx} admits a solution such that
$$
\left(\sum_{j=0}^{m-1} |\mathfrak{C}_j^{(X)}|^2\right)^{1/2} \le \frac{2}{|\lambda|_{\min}}
\left(\sum_{j=0}^{m-1} |h_{av}^{-p} (\epsilon_X(f,\hat{\Pi}))_{j}|^2\right)^{1/2},
$$
where $|\lambda|_{\min}$ is a minimal nonzero eigenvalue modulus of $L(\gamma, 0)$. The expressions $h_{av}^{-p}(\epsilon_X(f,\hat{\Pi}))_{j}$ do not depend on $h_{av}$ for a fixed $\gamma$, are equal to zero when $\gamma=0$, and are smooth in $\gamma$, thus their norms are bounded by $C|\gamma|$ as $|\gamma| \to 0$.

Let $\mathcal{F}_{\mu}$ be a set of meshes with $|\gamma|$ small enough, such that the conditions above hold. For $X \in \mathcal{F}_{\mu}$ define $\mathfrak{C}_j^{(X)}$ by \eqref{sys_for_cx}, and for $X \not \in \mathcal{F}_{\mu}$ put \mbox{$\mathfrak{C}_j^{(X)} = 0$}. The inequality $|\mathfrak{C}_j^{(X)}| \le c_1(h_{\max}-h_{\min})/h_{av}$ follows from $|\gamma| \le (h_{\max}-h_{\min})/h_{av}$.

The fact that the mapping \eqref{eq_tildepih} is local follows from the fact that $\mathfrak{C}_j^{(X)}$ do not depend on $h_{av}$ if $\gamma$ is fixed.
\end{proof}

For $q \in \mathbb{N}$ denote
$$
|v_0|_q = \sup\limits_{x \in \mathbb{R}}  \left|\frac{d^q v_0}{d x^q}(x)\right|.
$$

\begin{theorem}
Consider the scheme \eqref{eqFV1}--\eqref{eqFV3} with the polynomial reconstruction of order $p = 2s$, $s \in \mathbb{N} \cup \{0\}$.
For each $\delta>0$ there exists $\mu : \mathbb{N} \to (0,\infty)$ that depends on $p$ and $\delta$ only, such that for each mesh $X \in \mathcal{F}_{\mu}$ and for each $2\pi$-periodic \mbox{$v_0 \in C^{p+2}(\mathbb{R})$} the solution $u(t)$ of \eqref{eqFV1}--\eqref{eqFV3} satisfies
\begin{equation}
\begin{gathered}
\frac{1}{\sqrt{2\pi}}
\left(\sum\limits_{j=1}^{N(X)} h_{j+1/2} \Biggl|u_j(t) - \frac{1}{h_{j+1/2}} \int\limits_{x_{j}}^{x_{j+1}} v_0(x-t) dx \Biggr|^2\right)^{1/2} \le
\\
\le
 c |v_0|_{p+1} h_{\max}^{p} (h_{\max} - h_{\min}) + (c_s+\delta)|v_0|_{p+2} h_{\max}^{p+1} t,
\end{gathered}
\label{eq_est_FV}
\end{equation}
where $c$ depends only on $m(X)$ and $p$, and the constant $c_s$ is given by \eqref{eq_def_cs}.
\label{th:1}
\end{theorem}

On uniform meshes one may put $\delta=0$, then the right-hand side of  \eqref{eq_est_FV} reduces to $c_s |v_0|_{p+2}h_{\max}^{p+1}t$, which is the standard error estimate for the finite-difference scheme of order $p+1$. 

\begin{proof}
On the uniform meshes, the truncation error has the order $p+1$. Since $p+1 = 2s+1$ is odd, then $\varkappa = p+1$. On a general mesh, the scheme is $(p+1)$-exact in the sense of $\tilde{\Pi}$. Thus, the assumptions of Theorem~\ref{th:stab} hold. Hence, for each  $K>1$ there exists $\mu\ :\ \mathbb{N} \to (0,\infty)$, such that for each $X \in \mathcal{F}_{\mu}$ the scheme is stable with the constant $K$.

Let the mapping \eqref{eq_tildepih} with its coefficients $\mathfrak{C}_j^{(X)}$ be given by Lemma~\ref{th:aux:10}.
Let $\tilde{u}(t)$ be the solution of \eqref{eqFV1}, \eqref{eqFV3} with the initial data $\tilde{\Pi}_X v_0$. Then
\begin{equation}
\begin{gathered}
\|u(t) - \hat{\Pi}_X v(t, \ \cdot\ )\|_{av} \le
\\
\le
\|u(t) - \tilde{u}(t)\|_{av} + \|\tilde{u}(t) - \tilde{\Pi}_X v(t, \ \cdot\ )\|_{av}
+ \|\hat{\Pi}_X v(t, \ \cdot\ ) - \tilde{\Pi}_X v(t, \ \cdot\ )\|_{av}.
\end{gathered}
\label{eq_err1}
\end{equation}
The last term on the right-hand side of \eqref{eq_err1} has an estimate \mbox{$c_1 |v_0|_{p+1} h_{av}^{p} (h_{\max}-h_{\min})$}.
By stability, for the first term we have
\begin{equation*}
\begin{gathered}
\|u(t) - \tilde{u}(t)\|_{av} \le K \|u(0) - \tilde{u}(0)\|_{av} =
\\
= K \|\hat{\Pi}_X v(0, \ \cdot\ ) - \tilde{\Pi}_X v(0, \ \cdot\ )\|_{av}
\le Kc_1 |v_0|_{p+1} h_{av}^{p} (h_{\max}-h_{\min}).
\end{gathered}
\end{equation*}
The middle term in the right-hand side of \eqref{eq_err1} is the solution error by a stable scheme, thus
$$
\|\tilde{u}(t) - \tilde{\Pi}_X v(t, \ \cdot\ )\|_{av} \le Kt \max\limits_{0<t'<t} \|\epsilon_X(v(t', \ \cdot\ ), \tilde{\Pi})\|_{av}.
$$

To obtain an estimate for $(\epsilon_X(v(t', \ \cdot\ ), \tilde{\Pi}))_j$ for a fixed $t'$, write
$$
v(t', x) = p(x) + q(x),
$$
where $p(x)$ is the Taylor polynomial of order $p+1$ of the function $v(t',x)$ at  \mbox{$x=x_j$}. Then $|q(x)| \le (x-x_j)^{p+2} |v_0|_{p+2}/(p+2)!$. Since the scheme is $(p+1)$-exact by construction of $\tilde{\Pi}$,
$$
(\epsilon_X(v(t', \ \cdot\ ), \tilde{\Pi}))_j = (\epsilon_X(q, \tilde{\Pi}))_j.
$$
The right-hand side of the last equation can be estimated by definition. On the uniform mesh, it has the representation \eqref{eq_aux_2}. Since $\mathfrak{C}_j^{(X)} \rightarrow 0$ as $\vert\gamma\vert \rightarrow 0$, for each $\delta' > 0$ one can take a set $\mathcal{F}_{\mu}$ such that for $X \in \mathcal{F}_{\mu}$ there holds
$$
\|\epsilon_X(v(t', \ \cdot\ ), \tilde{\Pi})\|_{av} \le (c_s+\delta') h_{\max}^{p+1} |v|_{p+2}.
$$
Taking $\delta' = K^{-1} (c_s+\delta)-c_s$ (for $K$ close enough to 1, this is positive), we get
$$
\|\tilde{u}(t) - \tilde{\Pi}_X v(t, \ \cdot\ )\|_{av} \le (c_s + \delta) t h_{\max}^{p+1} |v|_{p+2}.
$$
Combining the estimates for each term in the right-hand side of  \eqref{eq_err1}, we get
\begin{equation}
\|u(t) - \hat{\Pi}_X v(t, \ \cdot\ )\|_{av}
\le c |v_0|_{p+1} h_{\max}^{p} (h_{\max} - h_{\min}) + (c_s+\delta)|v_0|_{p+2} h_{\max}^{p+1} t.
\label{eq_weqweqweqw2390123e}
\end{equation}
By choosing $\mu$ the ratio of the norms in the left-hand sides of \eqref{eq_est_FV} and \eqref{eq_weqweqweqw2390123e} can be made as close to 1 as necessary. Then the estimate \eqref{eq_est_FV} follows.
\end{proof}

\section{Example: meshes with alternating steps}
\label{sect:alternating}

To illustrate the results of the paper, we consider a mesh with alternating steps, i. e. a mesh with period $m=2$. So we assume that $N$ is even; let $\xi \in [0,1)$ and $\gamma = (\xi, -\xi)$ be the mesh structure. In  terms of the mesh nodes, $x_{k} = k h_{av}$ if $k$ is even and $x_{k} = (k+\xi) h_{av}$ if $k$ is odd. Then  $h_{\max} = (1+\xi) h_{av}$, $h_{\min} = (1-\xi) h_{av}$.

The schemes with the polynomial reconstruction are defined above. 
Here we also study schemes based on divided differences. We consider the scheme R3, which is the 1D version of the ``$\beta$-scheme'' in \cite{Dervieux2000} and the scheme EBR3 (edge-based reconstruction) in \cite{Bakhvalov2017CAF}. And we consider the scheme R5, which is the 1D version of ``Method 2'' of LV5 in \cite{Dervieux2000} and SEBR5 in \cite{Bakhvalov2017CAF, Bakhvalov2022b}. 


The schemes R3 and R5 have the form
\begin{equation}
\frac{du_j}{dt} + \frac{F_{j+1/2} - F_{j-1/2}}{(h_{j+1/2}+h_{j-1/2})/2} = 0,
\label{eqR1}
\end{equation}
\begin{equation}
u_j(0) = v_0(x_j),
\label{eqR2}
\end{equation}
where for the scheme R3 we set
\begin{equation}
F_{j+1/2} = u_j + \frac{h_{j+1/2}}{2} \left(\frac{2}{3} \frac{u_{j+1} - u_j}{h_{j+1/2}} + \frac{1}{3} \frac{u_{j}-u_{j-1}}{h_{j-1/2}}\right)
\label{eqR3}
\end{equation}
and for the scheme R5 we take
\begin{equation}
\begin{gathered}
F_{j+1/2} = u_j + \frac{h_{j+1/2}}{2} \left(-\frac{1}{10} \frac{u_{j+2}-u_{j+1}}{h_{j+3/2}} + \frac{4}{5} \frac{u_{j+1} - u_j}{h_{j+1/2}} +
\right.
\\
\left.
+ \frac{11}{30} \frac{u_{j}-u_{j-1}}{h_{j-1/2}} - \frac{1}{15}\frac{u_{j-1} - u_{j-2}}{h_{j-3/2}}\right).
\end{gathered}
\label{eqR5}
\end{equation}
The coefficients in front of divided differences are taken to obtain the maximal order of accuracy on uniform meshes, namely, the 3-rd order for R3 and the 5-th order for R5. 

The schemes R3 and R5 are 1-exact in the sense of the mapping $(\Pi_X f)_j = f(x_j)$. However, for each of them there exist a mapping $\tilde{\Pi}$ of the form $(\tilde{\Pi}_X f)_j = f(x_j) + \mathfrak{C}_j f''(x_j)$ such that the scheme is 2-exact on a family of meshes of the form \eqref{eq_def_fmu}. The proof of this fact is in analogy with Lemma~\ref{th:aux:10}. So the scheme R3 satisfies the conditions of Theorem~\ref{th:stab} (2-exact, $\varkappa = 3$), and the scheme R5 does not (2-exact, $\varkappa = 5$).

For the stability analysis, it is convenient to use the form \eqref{eq_def_L_gamma_phi}, \eqref{Fourier_Im_Eq}. The scheme is stable iff
$$
\sup\limits_{\phi \in \mathbb{R}} \sup\limits_{\nu \ge 0} \|\exp(-\nu L(\gamma,\phi))\| < \infty.
$$

\subsection{Scheme based on second-order polynomials}

For the scheme based on the second order polynomial reconstruction, the matrices $L_{\zeta}(\gamma)$ in \eqref{eq_generalform} are
$$
L_{-1} = \frac{1}{2(9-\xi^2)}\left(\begin{array}{cc}
3-4\xi+\xi^2 & -18+10\xi \\
0 & 3+4\xi+\xi^2
\end{array}\right),
$$
$$
L_0 = \frac{1}{2(9-\xi^2)}\left(
\begin{array}{cc}
9-8\xi-\xi^2 & 6+2\xi \\
-18-10\xi &  9+8\xi-\xi^2
\end{array}
\right),
$$
$$
L_1 = \frac{1}{2(9-\xi^2)} \left(\begin{array}{cc}
0 & 0 \\
6-2\xi & 0
\end{array}\right).
$$
The matrix $L(\gamma,\phi)$ defined by \eqref{eq_def_L_gamma_phi} has two eigenvalues:
$$
\lambda_*(\phi) = \frac{12}{9-\xi^2} + O(\phi), \quad \lambda_0(\phi) = i\phi + \frac{1-\xi^2}{12} \phi^4 + O(\phi^5)
$$
as $\phi \rightarrow 0$. The numerical evaluation of the eigenvalues shows the condition $\mathrm{Re}\lambda(\phi) > 0$ holds for both eigenvalues, each $\phi \ne \pi k$ ($k \in \mathbb{Z}$), and each $\xi \in [0,1)$. Thus the scheme based on the second-order polynomial reconstruction is stable at least for all meshes with period $m=2$.

\subsection{Scheme R3}


Denote $\mathcal{H} = (1+\xi)/(1-\xi)$, then the coefficients of the block representation \eqref{eq_generalform} of the scheme R3 are
$$
L_{-1} = \frac{1}{6}\left(\begin{array}{cc}
\mathcal{H}^{-1} &
-4 - \mathcal{H} - \mathcal{H}^{-1} \\
0 & \mathcal{H}
\end{array}\right),
$$
$$
L_0 = \frac{1}{6}\left(
\begin{array}{cc}
2 + \mathcal{H} & 2\\
-4 - \mathcal{H} - \mathcal{H}^{-1} &
2 + \mathcal{H}^{-1}
\end{array}
\right), \
L_1 = \frac{1}{6} \left(\begin{array}{cc}
0 & 0 \\
2 & 0
\end{array}\right).
$$
The matrix $L(\gamma,\phi)$ has two eigenvalues:
$$
\lambda_*(\gamma,\phi) = \frac{4}{3 (1-\xi^2)} + O(\phi),
$$
$$
\lambda_0(\gamma,\phi) = i\phi + \frac{1}{3} i\xi^2 \phi^3
+ \frac{1}{12} (1-\xi^2) (1 - 4\xi^2) \phi^4 + O(\phi^5)
$$
as $\phi \rightarrow 0$. For $\xi < 1/2$ both eigenvalues have non-negative real parts near $\phi=0$.  A detailed analysis shows that for $\xi \le 1/2$ the condition $\mathrm{Re}\lambda(\phi) > 0$ holds for each $\phi \ne \pi k$ and both eigenvalues. However, if $\xi > 1/2$, in a punctured neighborhood of $\phi=0$ there holds $\mathrm{Re}\lambda_0(\phi) < 0$. This means that the R3 scheme is unstable. Moreover, the instability develops on low-frequency modes.

\subsection{Scheme R5}

The coefficients of the block representation \eqref{eq_generalform} of the scheme R5 are
$$
L_{-2} = \frac{1}{60}\left(\begin{array}{cc} 0 & -2 \\ 0 & 0 \end{array}\right), \quad
L_{-1} = \frac{1}{60}\left(\begin{array}{cc} 4 + 11 \mathcal{H}^{-1} & -38-11\mathcal{H}-11\mathcal{H}^{-1} \\ -2 & 4 + 11 \mathcal{H} \end{array}\right),
$$
$$
L_0 = \frac{1}{60}\left(\begin{array}{cc} 12+11\mathcal{H}-3\mathcal{H}^{-1} & 24+3\mathcal{H}+3\mathcal{H}^{-1} \\ -38-11\mathcal{H}-11\mathcal{H}^{-1} & 12+11\mathcal{H}^{-1}-3\mathcal{H} \end{array}\right),
$$
$$
L_{1} = \frac{1}{60}\left(\begin{array}{cc} -3\mathcal{H} & 0 \\ 24+3\mathcal{H}+3\mathcal{H}^{-1} & -3\mathcal{H}^{-1} \end{array}\right).
$$
The matrix $L(\gamma,\phi)$ has two eigenvalues:
$$
\lambda_*(\gamma,\phi) = \frac{16}{15 (1-\xi^2)} + O(\phi),
$$
$$
\lambda_0(\gamma,\phi) = i\phi + \frac{1}{3} i\xi^2 \phi^3
- \frac{5}{12} \xi^2 (1-\xi^2) \phi^4 + O(\phi^5)
$$
as $\phi \rightarrow 0$. Clearly, for $\xi \ne 0$, we have $\mathrm{Re} \lambda_0(\gamma,\phi) < 0$ in a punctured neighborhood of $\phi=0$. This means that the scheme R5 is unstable on any non-uniform mesh with period $m=2$.

All Taylor expansions in this section were obtained with the use of the Sage mathematics software.

\subsection{Numerical results}

Now we return to the finite-volume schemes with the polynomial reconstruction of order $p=2s$, $s \in \mathbb{N} \cup \{0\}$. Theorem~\ref{th:1} states its $(p+1)$-th order convergence. Although numerical data supporting this result was reported previously many times, here we present our numerical results for the sake of the reader.

Consider the Cauchy problem \eqref{eq_TE} with $v_0(x) = \sin x$ on meshes with alternating steps, which are defined in the beginning of this section and illustrated in Fig.~\ref{fig:checkerboard}. We consider the cases $p=2$ and $p=4$, and the cases $h_{\max}/h_{\min} \in \{1,2,3\}$. For the time integration, we use the 7-th order linear Runge~-- Kutta method with the Courant number 0.1, so the time integration error is negligible. At $t=1$, we evaluate the norm of the solution error defined by the left-hand side of \eqref{eq_est_FV}. 

The data collected in Table~1 and Table~2 confirm the $(p+1)$-th order convergence of the polynomial-based schemes.

\begin{table}[t]
\caption{\label{table:1}The norm of the solution error of the FV scheme based on the 2-nd order polynomials}
\begin{center}
\begin{tabular}{|c|c|c|c|c|c|c|}
\hline
$h_{av}$ & \multicolumn{2}{|c|}{$h_{\max}=h_{\min}$} & \multicolumn{2}{|c|}{$h_{\max}=2h_{\min}$} & \multicolumn{2}{|c|}{$h_{\max}=3h_{\min}$} \\
  & solution error & order & solution error & order & solution error & order  \\
\hline
$2\pi/20$ & $1.86 \cdot 10^{-3}$ & & $1.69 \cdot 10^{-3}$ & & $1.47 \cdot 10^{-3}$ & \\
$2\pi/40$ & $2.28 \cdot 10^{-4}$ & 3.03 & $2.04 \cdot 10^{-4}$ & 3.05 & $1.73 \cdot 10^{-4}$ & 3.08 \\
$2\pi/80$ & $2.86 \cdot 10^{-5}$ & 2.99 & $2.55 \cdot 10^{-5}$ & 3.00 & $2.16 \cdot 10^{-5}$ & 3.01 \\
$2\pi/160$ & $3.58 \cdot 10^{-6}$ & 3.00 & $3.19 \cdot 10^{-6}$ & 3.00 & $2.69 \cdot 10^{-6}$ & 3.00 \\
$2\pi/320$ & $4.47 \cdot 10^{-7}$ & 3.00 & $3.97 \cdot 10^{-7}$ & 3.00 & $3.36 \cdot 10^{-7}$ & 3.00 \\
\hline
\end{tabular}
\end{center}
\end{table}   

\begin{table}[t]
\caption{\label{table:2}The norm of the solution error of the FV scheme based on the 4-th order polynomials}
\begin{center}
\begin{tabular}{|c|c|c|c|c|c|c|}
\hline
$h_{av}$ & \multicolumn{2}{|c|}{$h_{\max}=h_{\min}$} & \multicolumn{2}{|c|}{$h_{\max}=2h_{\min}$} & \multicolumn{2}{|c|}{$h_{\max}=3h_{\min}$} \\
  & solution error & order & solution error & order & solution error & order  \\
\hline
$2\pi/20$ & $3.64 \cdot 10^{-5}$ & & $3.64 \cdot 10^{-5}$ & & $3.60 \cdot 10^{-5}$ & \\
$2\pi/40$ & $1.12 \cdot 10^{-6}$ & 5.02 & $1.13 \cdot 10^{-6}$ & 5.01 & $1.14 \cdot 10^{-6}$ & 4.98 \\
$2\pi/80$ & $3.53 \cdot 10^{-8}$ & 4.99 & $3.62 \cdot 10^{-8}$ & 4.97 & $3.69 \cdot 10^{-8}$ & 4.94 \\
$2\pi/160$ & $1.10 \cdot 10^{-9}$ & 5.00 & $1.14 \cdot 10^{-9}$ & 4.99 & $1.17 \cdot 10^{-9}$ & 4.98 \\
$2\pi/320$ & $3.45 \cdot 10^{-11}$ & 5.00 & $3.57 \cdot 10^{-11}$ & 5.00 & $3.67 \cdot 10^{-11}$ & 4.99 \\
\hline
\end{tabular}
\end{center}
\end{table}

\section{Conclusion}

In the current study, we considered a general scheme on a non-uniform mesh. We proved a sufficient condition for the $L_2$-stability for a sufficiently small deformation of the computational mesh --- Theorem~\ref{th:stab}. Basing on this result, we proved the $(p+1)$-th order convergence of the high-order finite-volume schemes with the polynomial reconstruction.

We also verified our condition by investigating the stability of three schemes on the checkerboard meshes, which have periodic cells of two nodes. The scheme with the second-order polynomial reconstruction was found to be stable for a general $h_{\max}/h_{\min}$ ratio. The scheme R3 is stable if and only if $h_{\max}/h_{\min} \le 3$. And the scheme R5, which does not satisfy the assumptions of Theorem~\ref{th:stab}, was found to be unstable for any $h_{\max}/h_{\min} \ne 1$.

\end{document}